\documentclass[a4paper,10pt,reqno,american]{amsart}
\usepackage[utf8]{inputenc}
\usepackage[T1]{fontenc}
\usepackage{amssymb}
\usepackage{graphicx}
\usepackage{amsmath,amsthm}
\usepackage{amsfonts,amssymb,enumerate}
\usepackage{url,paralist}
\usepackage{anysize}
\usepackage{mathtools}
\usepackage{cancel}
\usepackage[arrow,curve,matrix,tips,2cell,rotate]{xy}
  \SelectTips{eu}{10} \UseTips
  \UseAllTwocells
\usepackage{lscape}
\usepackage{nccmath}
\usepackage{stackengine}

\usepackage[format=plain, font=footnotesize]{caption}
\usepackage{tikz}
\usepackage[colorlinks=true,urlcolor=blue,linkcolor=black,citecolor=magenta]{hyperref}
\usepackage{color}
\usepackage{enumerate,amssymb}
\usepackage{colortbl}

\makeindex

\theoremstyle{plain}
\newtheorem{theorem}{Theorem}[section]

\newtheorem{lemma}[theorem]{Lemma}

\newtheorem{corollary}[theorem]{Corollary}
\newtheorem{proposition}[theorem]{Proposition}

\newtheorem*{theorem*}{Theorem}

\theoremstyle{definition}

\newcommand{\BIGOP}[1]{\mathop{\mathchoice%
{\raise-0.22em\hbox{\huge $#1$}}%
{\raise-0.05em\hbox{\Large $#1$}}{\hbox{\large $#1$}}{#1}}}

\newcommand{\one}{\mathbf 1}

\newcommand{\R}{\mathbb{R}}

\newcommand{\X}{\mathcal{S}}
\newcommand{\XX}{\mathcal{X}}

\newcommand{\im}{\operatorname{im}}

\newcommand\Sym{\mathfrak S}

\newcommand{\conf}{\operatorname{F}} 

\newcommand{\A}{\mathrm{A}}

\newcommand{\spann}{\operatorname{span}}

\newcommand{\codim}{\operatorname{codim}}

\newcommand{\vol}{\operatorname{vol}}
\newcommand{\conv}{\operatorname{conv}}

\newcommand{\relint}{\operatorname{relint}}
\newcommand{\cone}{\operatorname{cone}}
\newcommand{\V}{\operatorname{VOL}}
\renewcommand{\emptyset}{\varnothing}
\newcommand{\Mat}{\operatorname{Mat}}
\newcommand{\rank}{\operatorname{rank}}
\newcommand{\diag}{\operatorname{diag}}
\newcommand{\trace}{\operatorname{trace}}

\begin{document}

\title[Facet volumes of polytopes]{Facet volumes of polytopes}


\author[Blagojevi\'c]{Pavle V. M. Blagojevi\'{c}}
\thanks{The research of Pavle V. M. Blagojevi\'{c} was supported by the Serbian Ministry of Education and Science  and the German Science Foundation DFG via the Collaborative Research Center TRR~109 ``Discretization in Geometry and Dynamics.''
}
\address{Mathemati\v{c}ki Institut SANU, Knez Mihailova 36, 11001 Beograd, Serbia}
\email{pavleb@mi.sanu.ac.rs}
\curraddr{\sc Institut f\"ur Mathematik, Freie Universit\"at Berlin, Arnimallee 2, 14195 Berlin, Germany}
\email{blagojevic@math.fu-berlin.de}

\author[Breiding]{Paul Breiding}
\thanks{The research of Paul Breiding was funded by the Deutsche Forschungsgemeinschaft (DFG, German Research Foundation), Projektnummer 445466444.}
\address{Max-Planck-Institute for Mathematics in the Sciences, Inselstr.\ 22, 04103 Leipzig, Germany}
\email{paul.breiding@mis.mpg.de}

\author[Heaton]{Alexander Heaton}
\thanks{The research of Alex Heaton was supported by the Fields Institute for Research in Mathematical Sciences}
\address{Lawrence University, Appleton, Wisconsin, USA}
\email{alexheaton2@gmail.com}

\date{}%
\renewcommand\medskip{\vspace{2pt}}

\begin{abstract}
In this paper, motivated by the work of Edelman and Strang, we show that for fixed integers $d\geq 2$ and $n\geq d+1$ the configuration space of all facet volume vectors of all $d$-polytopes in $\R^{d}$ with~$n$ facets is a full dimensional cone in $\R^{n}$.
In particular, for tetrahedra ($d=3$ and $n=4$) this is a cone over a regular octahedron.
Our proof is based on a novel configuration space / test map scheme which uses topological methods for finding solutions of a problem, and tools of differential geometry to identify solutions with the desired properties.
Furthermore, our results open a possibility for the study of realization spaces of all $d$-polytopes in $\R^{d}$ with $n$ facets by the methods of algebraic topology.

\end{abstract}

\maketitle



\section{Introduction and the statement of main results}
\label{Sec: Introduction and the statement of main results}

Already in elementary school we learn that each triangle in the plane is defined, up to a plane isometry, by the length of its edges.
These lengths fulfill three triangle inequalities.
On the other hand, every triple of positive real numbers, satisfying all three triangle inequalities, gives rise to a unique triangle, up to a plane isometry, with edge lengths coinciding with the given numbers.


\medskip
{\em What would be a high dimensional analogue of this basic school fact?}
Let $d\geq 1$ be an integer, and let $(x_1,\dots,x_{d+1})$ be a collection of $d+1$ affinely independent points in $\R^d$.
The convex hull of the collection
\[
\Delta_d(x_1,\dots,x_{d+1}):=\conv\{x_1,\dots,x_{d+1}\}\subseteq \R^d
\]
is a $d$-dimensional simplex in $\R^d$.
The simplex $\Delta_d(x_1,\dots,x_{d+1})$ has $d+1$ facets given by
\[
F_i(x_1,\dots,x_{d+1}):= \conv\big(\{x_1,\ldots,x_{d+1}\}\setminus\{x_i\}\big),
\]
where $1\leq i\leq d+1$.
Hence,  $F_i(x_1,\dots,x_{d+1})$ is the facet opposite to the vertex $x_i$.
Let us now consider the $(d-1)$-dimensional volume of the facets, denoted by
\[
v_i(x_1,\dots,x_{d+1}):=\vol_{d-1}\big(F_i(x_1,\dots,x_{d+1})\big).
\]
In this way we have defined a map $\nu_d\colon \A(\R^d,d+1)\to\R^{d+1}$ given by
\[
\nu_d(x_1,\dots,x_{d+1})\ =  \ \big(v_1(x_1,\dots,x_{d+1}),\dots, v_{d+1}(x_1,\dots,x_{d+1})\big),
\]
where $\nu_d(x_1,\dots,x_{d+1})$ is called the {\em facet volume vector} of the simplex $\Delta_d(x_1,\dots,x_{d+1})$.
Here $\A(\R^d,d+1)$ denotes the space of all collections of $d+1$ affinelly independent points in $\R^d$, or in other words,
\[
\A(\R^d,d+1):=\big\{ (x_1,\dots,x_{d+1})\in(\R^d)^{d+1} \ :\ \big(\forall (\lambda_1,\dots,\lambda_{d+1})\in W_{d+1}{\setminus}\{0\}\big) \sum_{1\leq i\leq d+1}\lambda_ix_i \neq 0 \big\},
\]
where $W_{d+1}:=\{(\lambda_1,\dots,\lambda_{d+1})\in \R^{d+1} : \sum_{1\leq i\leq d+1}\lambda_i=0\}.$
The space $\A(\R^d,d+1)$ is also known as the {\em realization space} of the $d$-dimensional simplex in $\R^d$.
Note that the definition of space $\A(\R^d,d+1)$ can be extended to all collections of $r$ points where $1\leq r\leq d+1$.
For example, $\A(\R^d,1)=\R^d$, and $\A(\R^d,2)$ coincides with the classical configuration space of two pairwise distinct points in $\R^d$.
Furthermore, the first facet volume can be expressed as $v_1(x_1,\ldots,x_{d+1}) = \tfrac{1}{(d-1)!}\sqrt{\det(A^t\cdot A)}$, where
$A$ is the matrix $[x_3-x_2 \ \cdots \ x_{d+1}-x_2]\in\R^{d\times (d-1)}$.
Analogous formulas hold for the other volumes.
This shows that $\nu_d$ is a smooth map, when $\A(\R^d,d+1)$ is endowed with the subspace topology from $(\R^d)^{d+1}$.

Now, our elementary school knowledge tells us that
\[
\im(\nu_1)=\{(0,0)\}
\ \quad\text{and} \ \quad
\im(\nu_2)=\big\{(\alpha_1,\alpha_2,\alpha_3)\in(\R_{>0})^3 \ :\ \alpha_1+\alpha_2<\alpha_3, \ \alpha_2+\alpha_3<\alpha_1, \ \alpha_3+\alpha_1<\alpha_2 \big\}.
\]
Here $\R_{>0}=(0,+\infty) \subseteq\R$ denotes the subset of the positive real numbers.
Hence, the general analogous question we want to answer is: {\em What is the image of the map $\nu_d$ for every $d\geq 1$, or in other words, what facet volume vectors can one get in an arbitrary dimension?}

\begin{theorem}
\label{Th: Main 01}
Let $d\geq 2$ be an integer, and let $\Sym_{d+1}$ denote the symmetric group on $d+1$ letters.
Then the space of all facet volume vectors is the open cone:
\begin{equation}
\label{Th: Main 01 - equality}
\V_{d,d+1}:=\im(\nu_d)=\big\{(\alpha_1,\dots,\alpha_{d+1})\in(\R_{>0})^{d+1} \ :\ \alpha_{\pi(d+1)}<\alpha_{\pi(1)}+\dots+ \alpha_{\pi(d)},\,\pi\in\Sym_{d+1}\big\}.
\end{equation}
\end{theorem}

\begin{figure}
  \begin{center}
\includegraphics[height = 4cm]{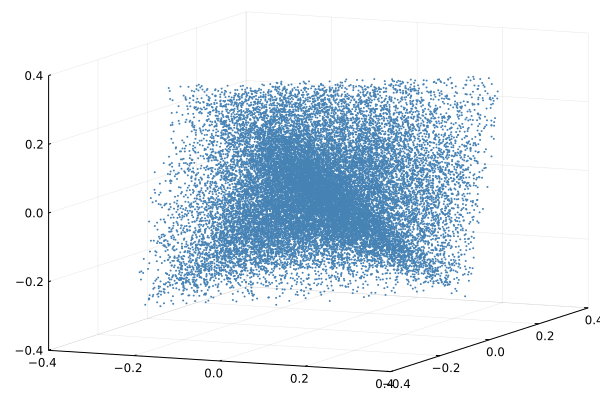}
\hfill
\includegraphics[height = 4cm]{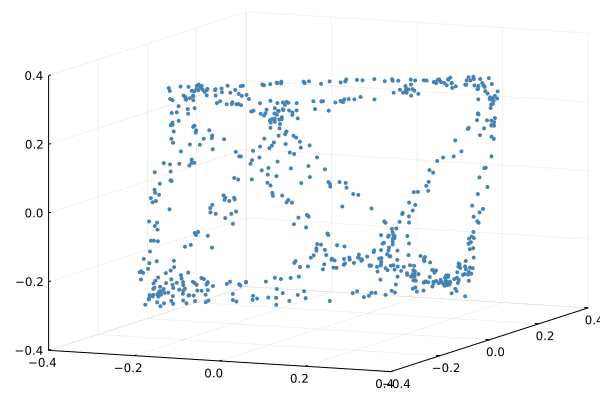}
\end{center}
\caption{\label{fig1} The picture on the left shows $5\cdot 10^5$ facet volume vectors of tetrahedra, which were obtained as $\nu_3(x_1,x_2,x_3,x_4)$, were $x_1,x_2,x_3,x_4\in\mathbb R^3$ are independently chosen random points with standard Gaussian entries. The right picture shows a subsample of points on the boundary. One can see that the points form a regular octahedron. This is proved in Corollary \ref{Cor : Main - 01}. The pictures were created using~\texttt{Plots.jl} \cite{plots}.}
\end{figure}

\medskip
The statement of the theorem uses the natural (left) action of the symmetric group $\Sym_{d+1}$ on $\R^{d+1}$ given by the permutation of the coordinates.
For this reason, the defining inequalities in the right hand side of \eqref{Th: Main 01 - equality} repeat many times, but on the other hand it is clear that $\V_{d,d+1}$ is an $\Sym_{d+1}$-invariant subspace of $\R^{d+1}$.
Furthermore, the map $\nu_d$ is $\Sym_{d+1}$-equivariant if the action on $\A(\R^d,d+1)$ is given by the permutation of the points --- the vertices of a simplex.

\medskip
The points on the diagonal $\{(\alpha_1,\dots,\alpha_{d+1})\in(\R_{>0})^{d+1} \ : \ \alpha_1=\dots=\alpha_{d+1}\}$ which belong to $\V_{d,d+1}$ correspond to the simplices whose facet volumes all coincide.
In the plane these are equilateral triangles.

\medskip
In their recent publication Alan Edelman and Gilbert Strang \cite{EdelmanStrang2015randomTriangleTheory} studied the shape of $\V_{2,3}=\im(\nu_2)$.
For example, they showed that when the edge lengths $(\alpha_1,\alpha_2,\alpha_3)$ of a triangle are normalised to the unit sphere, that is $\alpha_1^2+\alpha_2^2+\alpha_3^2=1$, and the point on the unit sphere $e=\tfrac{1}{\sqrt{3}}(1,1,1)$ is considered as the north pole, then latitudinal circles correspond to the triangles of equal area.

\medskip
Motivated by the idea of normalization, but now with respect to the affine hyperplane $H_{d+1}$ given by the equality $\sum_{1\leq i\leq d+1}\alpha_i=1$, instead of the unit sphere, we get the following description of the cone $\V_{d,d+1}$.
For this we take $e_1,\dots,e_{d+1}$ to be the standard basis of $\R^{d+1}$, set $e:=\sum_{1\leq i\leq d+1}e_i$, and denote by $\Delta_d$ the $d$-dimensional simplex $\conv\big\{e_1-\tfrac1{d+1}e,\dots, e_{d+1}-\tfrac1{d+1}e\big\}\subseteq W_{d+1}=H_{d+1}-\frac1{d+1}e$.
In the following, the standard scalar product of the Euclidean space $\R^{d+1}$ will be denoted by $\langle\cdot,\cdot\rangle$.

\begin{corollary}
\label{Cor : Main - 01}
Let $d\geq 2$ be an integer, $P_d:=\conv\big(\tfrac{2d+2}{d-1}\Delta_d \,\cup\, (d+1)(-\Delta_d)\big)\subseteq W_{d+1}$, and let $P_d^*$ be its polar in $W_{d+1}$.
Then,
\begin{equation}
\label{Th: Main 02 - equality}
\V_{d,d+1}\cap H_{d+1} = \relint (P_d^*) +  \tfrac1{d+1}e
\qquad\text{and}\qquad
\V_{d,d+1}= \cone  \big( \relint (P_d^*) +  \tfrac1{d+1}e\big){\setminus}\{0\}.
\end{equation}
\end{corollary}
\begin{proof}
From Theorem \ref{Th: Main 01} we know that $\V_{d,d+1}$ is the open cone
\[
\big\{\alpha=(\alpha_1,\dots,\alpha_{d+1})\in(\R_{>0})^{d+1} \ :\  \alpha_{1}+\dots+\alpha_{i-1}-\alpha_i+\alpha_{i+1}+\dots+ \alpha_{d+1}> 0,\,1\leq i\leq d+1\big\},
\]
which, using the scalar product, can also be presented by
\[
\big\{\alpha\in \R^{d+1} \ : \ \langle \alpha,e_i\rangle>0,\,\langle \alpha,e-2e_i\rangle>0,\,1\leq i\leq d+1\big\}.
\]
On the other hand, the hyperplane $H_{d+1}$ can be presented by $\{\alpha\in\R^{d+1} : \langle \alpha,e\rangle=1\}$.
Therefore,
\begin{align*}
\V_{d,d+1}\cap H_{d+1} &= \big\{\alpha\in \R^{d+1} \, :\, \langle \alpha,e\rangle=1,\, 0<\langle \alpha,e_i\rangle<\tfrac12,\,1\leq i\leq d+1\big\}	\\
&= \big\{\alpha\in H_{d+1} \, : \,  -\tfrac{1}{d+1}< \langle \alpha-\tfrac1{d+1}e,e_i-\tfrac1{d+1}e\rangle <\tfrac{d-1}{2d+2},\,1\leq i\leq d+1\big\}\\
&= \big\{\alpha\in H_{d+1} \, : \, \langle \alpha-\tfrac1{d+1}e,\tfrac{2d+2}{d-1}(e_i-\tfrac1{d+1}e)\rangle < 1,\\
&\hspace{26.5mm} \langle \alpha-\tfrac1{d+1}e,(d+1)(-(e_i-\tfrac1{d+1}e))\rangle < 1,\,1\leq i\leq d+1\big\}\\
&=\relint (P_d^*) +  \tfrac1{d+1}e,
\end{align*}
as we have claimed.
\end{proof}

\medskip
Edelman and Strang \cite{EdelmanStrang2015randomTriangleTheory} showed that in case of triangles, $\V_{2,3}=\im(\nu_2)$ is a cone over the relative interior of an equilateral triangle where the apex, the origin, is deleted.
This can also be deduced from Corollary \ref{Cor : Main - 01} by observing the interesting phenomena that $-3\Delta_2\subseteq 6\Delta_2$.
In the case of dimension $d=3$, the polytope $P_3=\conv\big(4\Delta_3 \,\cup\, 4(-\Delta_3)\big)$ is a cube and consequently the polar $P_3^*$ is a regular octahedron.

\medskip
Another interesting geometric observation can be made.
Let $P$ be a $d$-dimensional polytope with facets $F_1,\dots,F_n$.
The \emph{dihedral angle} between the facets~$F_i$ and~$F_j$ of the polytope $P$ is defined to be the angle $\theta_{i,j}: =\arccos\langle-u_i,u_j\rangle$ between the corresponding (unit) outer unit normals $u_i$ and $u_j$ to the facets $F_i$ and~$F_j$.
In his paper from 2003, Leng \cite{Leng2003} proved that any $d$-dimensional simplex~$\Delta\subseteq \mathbb R^d$ has at most~$\tfrac{1}{2}d(d-1)$ obtuse dihedral angles, where the lower bound is attained by a specific class of simplices.
Leng's proof shows that there exists a $d$-dimensional simplex whose obtuse dihedral angles are exactly~$\theta_{i,j}$ for~$2\leq i<j\leq d+1$.
In such a simplex there is a facet which is not involved in any of its $\tfrac{1}{2}d(d-1)$ obtuse dihedral angles.
The facet volumes of this simplex satisfy the inequality
\[
\alpha_1^2 \geq \sum_{2\leq j\leq d} \alpha_j^2.
\]
Moreover, other collections of dihedral angles can also imply this inequality.
Motivated by the relationship between obtuse angles of a simplex and its facet volumes, we call a $d$-simensional simplex $\Delta$ \textit{obtuse} if for some index $1\leq i\leq d+1$ it holds that $\alpha_i^2 \geq \sum_{1\leq j\leq d+1, \,j \neq i} \alpha_j^2$.
Here $\alpha_1,\dots,\alpha_{d+1}$ denote, as before, the facet volumes of the corresponding simplex $\Delta$.
If such an index does not exist, a simplex is called \textit{acute}.
It is important to mention that our definition differs from other definitions of acute and obtuse simplices in the literature.
For example, Eppstein et.\ al. \cite{EPPSTEIN2004237} define a tetrahedron to be acute if all of its dihedral angles are acute.
This is a stronger condition.
Extension of this definition to arbitrary $d$-dimensional simplices was made by K\v{r}\'{i}\v{z}ek \cite{krizek2010}.

\medskip
The relationship between obtuse dihedral angles of a simplex and its facet volumes motivates yet another parametrization of simplices by, now, squared facet volumes.
We consider the set
\begin{equation*}
    \V_{d,d+1}^2 := \{ (\alpha_1^2,\dots, \alpha_{d+1}^2) \in \R^{d+1} \, : \, (\alpha_1, \dots, \alpha_{d+1}) \in \V_{d,d+1} \},
\end{equation*}
and the subspace $\mathcal{A} \subset \V_{d,d+1}^2$ of all acute simplices; these are the simplices whose facet volumes satisfy the inequalities $\alpha_i^2 < \sum_{1\leq j\leq d+1,\,j \neq i} \alpha_j^2$ for all $1\leq i\leq d+1$.
Then with $P_d$ and $H_{d+1}$, as defined in Corollary \ref{Cor : Main - 01}, we can show the following fact.
\begin{corollary}
$\mathcal{A} \cap H_{d+1} = \relint (P_d^*) +  \tfrac1{d+1}e$.
\end{corollary}

\medskip
In the case of a plane this says that acute triangles form an equilateral triangle, while in the case $d=3$ we have that acute tetrahedra form a regular octahedron.

\begin{figure}
  \begin{center}
\includegraphics[height = 4cm]{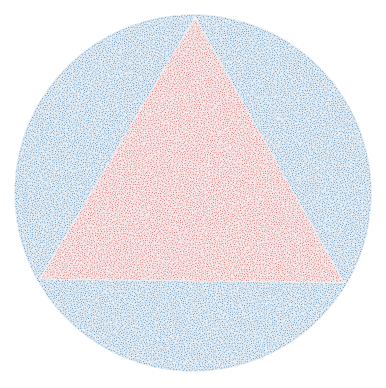}
\hspace{2em}
\includegraphics[height = 4cm]{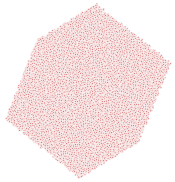}
\hspace{2em}
\includegraphics[height = 3.8cm]{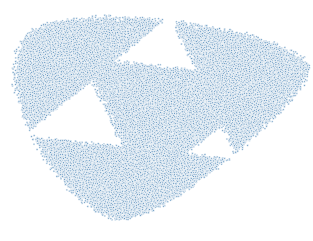}
\end{center}
\caption{\label{fig2} The picture on the left is similar to \cite[Figure 4]{EdelmanStrang2015randomTriangleTheory}. It shows $5\cdot 10^5$ triangles  in the coordinates $s_1,s_2,s_3$, the squared side lengths, subject to $s_1+s_2+s_3=1$. The data was sampled as $\nu_2(x_1,x_2,x_3)$, were $x_1,x_2,x_3\in\mathbb R^2$ are independently chosen random vectors with standard Gaussian entries. The red points are acute triangles. The blue points are obtuse triangles. One can see that the acute triangles form an equilateral triangle. The other two pictures show  $5\cdot 10^5$ tetrahedra in the coordinates $s_1,s_2,s_3,s_4$, squared facet areas, in the plane $s_1+s_2+s_3+s_4=1$.
The red points are acute tetrahedra and form a regular octahedron. The octahedron's absence can be seen in the right picture, which shows obtuse tetrahedra in blue. Merging the two pictures would give a single convex body of all tetrahedra, obtuse and acute. The pictures were created using \texttt{Plots.jl} \cite{plots}.}
\end{figure}

\medskip
After complete determination of the space of facet volume vectors of $d$-dimensional simplices in $\R^d$ a natural question arrises: {\em What about facet volume vectors of other polytopes?}
More precisely, for given integers $d\geq 2$ and $n\geq d+1$, what is the space $\V_{d,n}$ of all vectors $(\alpha_1,\dots,\alpha_n) \in(\R_{>0})^{n}$ such that there exists a $d$-dimensional polytope $P$ in $\R^d$ with $n$ facets $F_1,\dots,F_n$ with the property that $\alpha_i=\vol_{d-1}(F_i)$ for every $1\leq i\leq n$.
As in the case of simplices ($n=d+1$), the set $\V_{d,n}$ can also be seen as an image of the realization space of all $d$-polytopes with $n$-facets along the appropriate map.

\medskip
The set of all volume vectors $\V_{d,n}$ can be described in general, and we show the following generalization of Theorem \ref{Th: Main 01}.

\begin{theorem}
\label{Th: Main 02}
Let $d\geq 2$ and $n\geq d+1$ be integers.
Then,
\[
\V_{d,n}=\big\{ (\alpha_1,\dots,\alpha_n) \in(\R_{>0})^{n} \, : \, \alpha_{\pi(n)}<\alpha_{\pi(1)}+\dots+ \alpha_{\pi(n-1)},\,\pi\in\Sym_{n}\big\}.
\]

\end{theorem}

In the statement of Theorem \ref{Th: Main 02} we decide to neglect the description of the cone $\V_{d,n}$ as the image of a realization space of all $d$-dimensional polytopes in $\R^{d}$ with $n$ facets.
Hence, we do not elaborate on the different presentations of the realization space which can be found in the literature.
For more details on realization spaces consult for example the recent publication of Rastanawi, Sinn and Ziegler \cite{RastanawiSinnZiegler2021}.

\medskip
Theorem \ref{Th: Main 01} is a special case of Theorem \ref{Th: Main 02}.
Hence, in Section \ref{Sec: Facet volume vectors of polytopes}, we prove only Theorem  \ref{Th: Main 02} by combining classical results about polytopes with methods from differential geometry and topology. In fact, the case $d=2$ was also proven by Manecke and Sanyal in \cite[Proposition 3.8]{manecke2020inscribable}. For completeness, we include this case in our proof.

\medskip
The normalization, in the general case of a $d$ polytope with $n\geq d+1$ facets, with respect to the hyperplane $H_{n}:=\{(\alpha_1,\dots,\alpha_n)\in\R^n : \sum_{1\leq i\leq n}\alpha_i=1\}$, yields the following corollary.

\begin{corollary}
\label{Cor : Main - 02}
Let $d\geq 2$ and $n\geq d+1$ be an integer, $P_{n-1}:=\conv\big(\tfrac{2n}{n-2}\Delta_{n-1} \,\cup\, n(-\Delta_{n-1})\big)\subseteq W_{n}$, and let $P_{n-1}^*$ be its polar in $W_{n}$.
Then,
\begin{equation}
\label{Th: Main 03 - equality}
\V_{d,n}\cap H_{n} = \relint (P_{n-1}^*) +  \tfrac1{n}e
\qquad\text{and}\qquad
\V_{d,n}= \cone  \big( \relint (P_{n-1}^*) +  \tfrac1{n}e\big){\setminus}\{0\},
\end{equation}
where $e:=\sum_{1\leq i\leq n}e_i$, and $e_1,\dots,e_{n}$ is the standard basis of $\R^{n}$.
\end{corollary}

\medskip
The statement can be verified along the lines of the proof of Corollary \ref{Cor : Main - 01} where instead of Theorem~\ref{Th: Main 01} we use rather its generalization, Theorem \ref{Th: Main 02}.

\medskip
{\em What comes next?}
The fact that facet volumes of polytopes form a cone introduces a geometrically motivated stratification of the realization space of all $d$-dimensional polytopes in $\R^{d}$ with $n$ facets.
A strata of a realization space is defined to be a fiber of the facet volume map. For example, the facet volume map $\nu_2\colon \A(\R^2,3)\to \V_{2,3}$ induces a bijection between all affine isometry classes of triangles in the plane and the cone $\V_{2,3}$, that is~$\A(\R^2,3)/\mathrm{(O(2)\times \R^2)} \cong \V_{2,3}$.
Here~$\mathrm{O}(2)\times \R^2$ denotes the group of the affine Euclidean isometries in $\R^2$.


\medskip
Consider a facet volume map, from $\mathcal{R}_{d,n}$, the realization space  of all $d$-dimensional polytopes in $\R^d$ with $n$ facets, onto the cone $\V_{d,n}$, and its induced maps
\[
\mathcal{R}_{d,n}\ \twoheadrightarrow\ \mathcal{R}_{d,n}/(\mathrm{O}(d)\times \R^d)\ \twoheadrightarrow\ \V_{d,n},
\]
and
\[
\mathcal{R}_{d,n}\ \twoheadrightarrow\  \mathcal{R}_{d,n}/(\mathrm{O}(d)\times \R^d)\ \twoheadrightarrow\ \V_{d,d+1} \ \twoheadrightarrow\  \V_{d,n}/\Sym_n.
\]
The questions we should ask are:
What are the fibers of all, or any of, the maps in the compositions?
More concretely, what is the topology of strata of the realization space $\mathcal{R}_{d,n}$ or its orbit space $\mathcal{R}_{d,n}/(\mathrm{O}(d)\times \R^d)$?
A pioneered work, related to the proposed questions, for the case of arbitrary, not necessarily convex, polygons, considered up to a positive similarity, was done in 1998 by Kapovich and Millson \cite{KM1998}.

\bigskip
\noindent\textbf{Acknowledgements}
The authors would like to thank Raman Sanyal and and G\"unter Ziegler for several helpful discussions.

\section{Facet volume vectors of polytopes: Proof of Theorem \ref{Th: Main 02} }
\label{Sec: Facet volume vectors of polytopes}

Let $d\geq 2$ and $n\geq d+1$ be integers.
As before, let $e_1,\dots,e_{n}$ denote the the standard basis of $\R^{n}$ and set $e:=\sum_{1\leq i\leq n}e_i$.
As before, $\langle\cdot,\cdot\rangle$ stands for the standard scalar product of the Euclidean space $\R^{n}$.

\medskip
In this section we prove Theorem \ref{Th: Main 02} by showing that a vector $\alpha=(\alpha_1,\dots,\alpha_n) \in(\R_{>0})^{n}$ satisfies the system of inequalities
\[
\langle \alpha,e-2e_1\rangle>0, \ \dots\, , \  \langle \alpha,e-2e_n\rangle>0
\]
if and only if there exists a $d$-dimensional polytope $P$ in $\R^d$ with $n$ facets $F_1,\dots,F_n$ such that
\[
\alpha_1=\vol_{d-1}(F_1),\ \dots\, , \ \alpha_n=\vol_{d-1}(F_n).
\]
The proof of this equivalence has two natural parts. Sufficiency is considerably simpler to prove, while necessity turns out to be more challenging.

\subsection{Sufficiency}
\label{Subsec: Sufficiency}
Let $P$ be a $d$-dimensional polytope in $\R^d$ with $n$ facets $F_1,\dots,F_n$.
Denote by $u_i$ the outer unit normal of the facet $F_i$, and set $\alpha_i:=\vol_{d-1}(F_i)$, for $1\leq i\leq n$.
Hence,
$$F_i=P\cap\{ x\in \R^d : \langle x,u_i\rangle = h_P(u_i)\},$$
where $h_P\colon\R^d\to\R$ is the support function of the polytope $P$.
For basic properties of the support function of a convex body consult for example \cite[Section\,1.7.1]{Schneider1993}.

\medskip
A classical result of convex geometry relates the vectors $u_i$ with scalars $\alpha_i$ through the linear dependence
\begin{equation}
\label{relation between u_i's and y_i's -- 01}
	\sum_{1\leq i\leq n} \alpha_i u_i = 0,
\end{equation}
consult for example \cite[Lemma 5.1.1]{Schneider1993}.
For any fixed index $1\leq j\leq n$ the equality \eqref{relation between u_i's and y_i's -- 01} can be transformed into
$
\alpha_ju_j=-\sum_{1\leq i \leq n, i\neq j} \alpha_i u_i.
$
Applying the triangle inequality, and having in mind that $\alpha_i>0$ for all $1\leq i\leq n$, we get that
\begin{equation}
\label{relation between u_i's and y_i's -- 02}
\alpha_j\|u_j\| < \sum_{1\leq i \leq n, i\neq j}\alpha_i\|u_i\|.
\end{equation}
The inequality is strict because all the unit outer normal vectors of an arbitrary $d$-dimensional polytope in $\R^d$ are pairwise distinct.
Finally, we recall that all the vectors $u_i$ are unit vectors, and consequently the inequality \eqref{relation between u_i's and y_i's -- 02} becomes:
\[
\alpha_j  < \sum_{1\leq i \leq n, i\neq j}\alpha_i \ \Longleftrightarrow \ \langle \alpha,e-2e_j\rangle>0.
\]
Thus, we completed the sufficiency part of the proof.

\subsection{Necessity}
\label{Subsec: Necessity}
Let $\alpha=(\alpha_1,\dots,\alpha_n) \in(\R_{>0})^{n}$ be a given vector which satisfies the inequalities
\begin{equation}
\label{system of inequalities}
\langle \alpha,e-2e_1\rangle>0, \ \dots\, , \  \langle \alpha,e-2e_n\rangle>0.
\end{equation}
We prove that there exists a $d$-dimensional polytope $P$ in $\R^d$ with $n$ facets $F_1,\dots,F_n$ having the property that $\alpha_i=\vol_{d-1}(F_i)$ for every $1\leq i\leq n$.
For this it suffices to show that there exists a collection of pairwise distinct unit vectors $(u_1,\dots, u_n)$ which linearly span $\R^d$ and satisfy the linear relation
$
\sum_{1\leq i\leq n} \alpha_i u_i = 0
$.
Indeed, according to the Minkowski's existence theorem \cite[Theorem\,8.2.1]{Schneider1993}, for any collection of pairwise distinct unit vectors $(u_1,\dots, u_n)$ which linearly span $\R^d$ and a collection of positive real numbers $(\alpha_1,\dots, \alpha_n)$ satisfying the relation $\sum_{1\leq i\leq n} \alpha_i u_i = 0$ and \eqref{system of inequalities} there exists a $d$-dimensional polytope~$P$ with $n$ facets $F_1,\dots,F_n$ such that $u_i$ is the outer unit normal of the facet $F_i$ and $\alpha_i=\vol_{d-1}(F_i)$, for every $1\leq i\leq n$.

\medskip
Thus, we complete the proof of the necessity part by proving the following proposition.
Recall that the vector $\alpha=(\alpha_1,\dots,\alpha_n) \in(\R_{>0})^{n}$ is already fixed and satisfies the system of inequalities \eqref{system of inequalities}.

\begin{proposition}
\label{proposition necessity}
There exists a collection of pairwise distinct unit vectors $(u_1,\dots, u_n)\in (\R^d)^n$ with
\[
\spann\{u_1,\dots, u_n\}=\R^d\qquad\quad\text{and}\qquad\quad \sum_{1\leq i\leq n} \alpha_i u_i = 0.
\]
\end{proposition}

The proof of the proposition combines different ideas and techniques and therefore is divided into several steps which follow.
Before going into the proof let us note that the case of a triangle, that is $d=2$ and $n=3$, is satisfied.
Thus, in the following we assume that $d\geq 3$ when $n=d+1$.

\subsubsection{Space of solution candidates}

Let $1\leq i\leq n$ be an integer, and let $S_i$ denotes the sphere in $\R^d$ of radius $\alpha_i$, that is $S_i:=\{x\in\R^d \ : \  \langle x,x\rangle = \alpha_i^2\}$.
Further on, we consider the product of $n-1$ spheres
\[
\X:=S_1\times\dots\times S_{n-1}
\]
as a $(d-1)(n-1)$-dimensional submanifold of $(\R^d)^{n-1}$.
The tangent space of the product of spheres $\X$, at the point $x=(x_1,\ldots,x_{n-1})\in \X$, can be presented by
\[
\tau_x\X = \big\{y=(y_1,\ldots,y_{n-1})\in (\R^d)^{n-1} \ : \ \langle x_i,y_i\rangle=0,\, 1\leq i\leq n-1\big\}.
\]
Let us now see the vector space $(\R^d)^{n-1}$ as the space of all $d\times (n-1)$ matrices with real entries $M:=\Mat_{d\times (n-1)}(\R)$.
If the point $x\in \X$ is presented as the columns of the $d\times (n-1)$ matrix $X$ then, the tangent space of $\X$ at $X$ is given by
\begin{equation}
\label{tangent space 01}
\tau_X\X = \big\{Y\in M \ : \ (X^t\cdot Y)_{i,i}= 0,\, 1\leq i\leq n-1\big\},
\end{equation}
where $(X^t\cdot Y)_{i,j}$ denotes the $(i,j)$-entry of the matrix $X^t\cdot Y$.

\medskip
Next, we consider the variety of all singular matrices $M'$ in $M$, that is
\begin{equation}\label{definition corank 1 matrices}
M':=\{A\in M \ :\ \rank (A)\leq d-1\}.
\end{equation}
The variety $M'$ can be partitioned into a family of smooth manifolds which are given by fixing the rank, these are manifolds
\begin{equation}\label{definition rank r matrices}
M_r:=\{A\in M : \rank (A)=r \}
\end{equation}
for $0\leq r\leq d-1$.
Hence, $M'=\coprod_{0\leq r\leq d-1}M_r$ is a stratification of $M'$ into smooth manifolds.
Each $M_r$ is a smooth manifold of dimension $(d+n-1-r)r$ embedded in $M$.
In addition, we denote by $M''$ the singular part of the variety $M'$, that is
\begin{equation}\label{definition corank 2 matrices}
M'':=\{A\in M \ : \ \rank (A)\leq d-2\}=\coprod_{0\leq r\leq d-2}M_r.
\end{equation}
For more details on the manifolds of fixed rank consult for example  \cite[Section\,2.1]{Vandereycken2013} or  \cite[Sec.\,2.2]{UschmajewVandereycken2020}.

\medskip
Let $0\leq r\leq d-1$ be an integer, and let $X\in M_r$ be a fixed matrix.
We describe the tangent space $\tau_XM_r$ as it was done in \cite[Section\,2.1]{Vandereycken2013}.
The singular value decomposition of a rank $r$ matrix gives a presentation of $X$ in the form $U\cdot D\cdot V^t$ where
\begin{compactenum}[\rm (i)]
\item $U\in\Mat_{d\times r}(\R)$ is a matrix of rank $r$ with the property that $U^t\cdot U=\one_r$,
\item  $D=\diag(d_1,\dots,d_r)\in\Mat_{r\times r}(\R)$ is a diagonal matrix with diagonal entries $d_1\geq \dots\geq d_r>0$, and
\item  $V\in\Mat_{(n-1)\times r}(\R)$ is also a matrix of rank $r$ with the property that $V^t\cdot V=\one_r$.
\end{compactenum}
Here $\one_r$ denotes the $r\times r$ unit matrix.
The manifold $M_r$ can be described as follows:
\begin{multline*}
M_r=\big\{U\cdot D\cdot V^t \ : \ U\in\Mat_{d\times r}(\R), \, \rank(U)=r, \, U^t\cdot U=\one_r,
\\ V\in\Mat_{(n-1)\times r}(\R), \, \rank(V)=r, \, V^t\cdot V=\one_r, \\ D=\diag(d_1,\dots,d_r)\in \Mat_{r\times r}(\R),\, d_1\geq d_2\geq \dots\geq d_r>0	\big\}.
\end{multline*}
Then, for $X=U\cdot D\cdot V^t\in M_r$, as explained in \cite[Proposition\,2.1]{Vandereycken2013},  we have that
\begin{multline}\label{tangent space 02}
\tau_XM_r
=
\big\{
U\cdot Z\cdot V^t + U_p\cdot V^t + U\cdot V_p^t \ : \ Z\in\Mat_{r\times r}(\R), \\ U_p\in\Mat_{d\times r}(\R),\, U_p^t\cdot U=0,\,   V_p\in\Mat_{(n-1)\times r}(\R),\, V_p^t\cdot V=0
\big\} .
\end{multline}

\medskip
In the first auxiliary lemma we prove that the submanifolds $\X$ and $M_r$, $0\leq r\leq d-1$, intersect transversally, in symbols $\X\pitchfork M_r$.
This means that for every point in the intersection $X\in \X\cap M_r$ the sum of vector spaces $\tau_X\X+\tau_XM_r$, where $X$ is now considered as the origin, coincides with the tangent vector spaces $\tau_XM$ to the ambient manifold $M=(\R^{d})^{n-1}$ at the point $X$.
For more details of transversal intersections see for example \cite[Chapter\,6, p.\,143]{Lee2013}.

\begin{lemma}
\label{lemma : transversal}
$\X\pitchfork M_r$, for all $1\leq r\leq d-1$.
\end{lemma}
\begin{proof}
Let  $X\in \X\cap M_r$, with $X=U\cdot D\cdot V^t$ for some $U\in\Mat_{d\times r}(\R)$, $V\in\Mat_{(n-1)\times r}(\R)$ and $D=\diag(d_1,\dots,d_r)\in\Mat_{r\times r}(\R)$, where $\rank(U)=r$, $\rank(V)=r$,  $U^t\cdot U=\one_r$,  $V^t\cdot V=\one_r$, and $d_1\geq d_2\geq \dots\geq d_r>0$.
We have seen in \eqref{tangent space 01} and \eqref{tangent space 02} than the tangent spaces to $S$ and $M_r$ at the point~$X$ can be described by:
\[
\tau_X\X = \big\{Y\in M \ : \ (X^t\cdot Y)_{i,i}= 0,\, 1\leq i\leq n-1\big\},
\]
and
\begin{align}\label{Tangent space rank r matrices}
\tau_XM_r
=
\big\{
U\cdot Z\cdot V^t + U_p\cdot V^t &+ U\cdot V_p^t \ : \ Z\in\Mat_{r\times r}(\R), \\ &U_p\in\Mat_{d\times r}(\R),\, U_p^t\cdot U=0,\,   V_p\in\Mat_{(n-1)\times r}(\R),\, V_p^t\cdot V=0
\big\} .\nonumber
\end{align}
Thus, in order to prove that $\X\pitchfork M_r$ we need to check that
$\tau_XM =\tau_X\X+ \tau_XM_r$.
Here $\tau_XM$ is a $d(n-1)$ dimensional real vector space, since $M$ is also a $d(n-1)$ dimensional real vector space.
Recall, $\dim\tau_X\X =(d-1)(n-1)$ and $\dim \tau_XM =(d+n-1-r)r$.

\medskip
Further on, let us denote by $E_{i}\in\Mat_{(n-1)\times (n-1)}(\R)$ the matrix with all entries zero except the entry $(i,i)$ which is assumed to be $1$.
We show that $X\cdot E_i\in \tau_x M_r$ for every $1\leq i\leq n-1$.
For that we note that $P_V:=V\cdot V^t$, considered as the linear map $\R^{n-1}\to \R^{n-1}, \ x\mapsto P_V\cdot x$, is a projection onto the subspace~$C$ spanned by the columns of the matrix $V$.
Indeed, $P_V\cdot V=(V\cdot V^t)\cdot V= V\cdot (V^t\cdot V)=V\cdot\one_r=V$, meaning that $P_V(x)=x$ for $x\in C$ , and for $x\in C^{\perp} \Leftrightarrow V^t\cdot x=0$ we have that $P_V(x)=(V\cdot V^t)\cdot x=V\cdot (V^t\cdot x)=0$.
We set:
\begin{align*}
Z&:= D\cdot V^t\cdot E_i\cdot V\in\Mat_{r\times r}(\R), \\ U_p&:=0\in \Mat_{d\times r}(\R), \text{ and}\\
V_p^t&:=D\cdot V^t\cdot E_i\cdot (\one_{n-1}-P_V)\in \Mat_{r\times (n-1)}(\R).
\end{align*}
Then, $U_p^t\cdot U=0$, and also $V_p^t\cdot V=D\cdot V^t\cdot E_i \cdot(\one_{n-1}-P_V)\cdot V = D\cdot V^t\cdot E_i \cdot(V-V)=0$.
Furthermore,
\begin{align*}
U\cdot Z\cdot V^t + U_p\cdot V^t + U\cdot V_p^t = U\cdot (Z\cdot V^t +  V_p^t)&= U\cdot (D\cdot V^t\cdot E_i \cdot V\cdot V^t +D\cdot V^t\cdot E_i \cdot (\one_{n-1} -P_v))\\
	&=U\cdot D\cdot V^t\cdot E_i \cdot(P_V+\one_{n-1}-P_V)\\
  &=(U\cdot D\cdot V^t)\cdot E_i =X\cdot E_i.
\end{align*}
Hence, $X\cdot E_i\in \tau_XM_r$.

\medskip
Let us present $X$ as the collection of its column vectors, say $X=[x_1 \ \dots \ x_{n-1}]$.
Hence, $X\cdot E_i$ is the matrix with all columns zero except the $i$th column which is $x_i$, that means $X\cdot E_i =[0 \ 0 \ \dots 0 \ x_i \ 0\ \dots \ 0]$.
Each vector $x_i$ belongs to the sphere $S_i$, and so cannot be zero implying that
\[
\dim\big(\spann\{X\cdot E_1,\dots, X\cdot E_{n-1}\}\big) = n-1.
\]
The subspaces $\tau_X\X$ and $\spann\{X\cdot E_1,\dots, X\cdot E_{n-1}\}\subseteq  \tau_XM_r$ are orthogonal with respect to the scalar product defined on $\tau_XM$ by $\langle\langle A,B\rangle\rangle:=\trace (A^t\cdot B)$.
Therefore,
\[
\tau_XM \ \supseteq \ \tau_X\X+ \tau_XM_r \ \supseteq \ \tau_X\X + \spann\{X\cdot E_1,\dots, X\cdot E_{n-1}\}\cong\tau_X\X \oplus \spann\{X\cdot E_1,\dots, X\cdot E_{n-1}\}.
\]
Since,
\begin{align*}
\dim
\big( \tau_X\X \oplus \spann\{X\cdot E_1,\dots, X\cdot E_{n-1}\}
\big)
&=
\dim (\tau_X\X)+\dim \big(\spann\{X\cdot E_1,\dots, X\cdot E_{n-1}\}
\big)
\\
&= (d-1)(n-1)+(n-1)\\
&=d(n-1)=\dim (\tau_XM),
\end{align*}
we proved that $\tau_XM =\tau_X\X+ \tau_XM_r$, and consequently $\X\pitchfork M_r$.
\end{proof}

\medskip
As a direct consequence of the previous lemma we determine the codimension of the each intersection~$\X \cap M_r$ in $\X$.

\begin{corollary}
\label{corollary 02}
$\codim_{M}(\X\cap M_r)=(d-r)(n-r-1)$, for all $1\leq r\leq d-1$.
\end{corollary}
\begin{proof}
According to Lemma \ref{lemma : transversal} we have seen that $\X\pitchfork M_r$, and so by \cite[Theorem\,6.30]{Lee2013} follows that
\(
\codim_M(\X\cap M_r)=\codim_M(\X)+\codim_M(M_r)=(d-r)(n-r-1).
\)
\end{proof}

The complement $\X{\setminus}M'$ is the space of all collections of $n-1$ vectors $(x_1,\dots, x_{n-1})$ which span the ambient vector space $\R^d$ and each vector is of the prescribed norm $\|x_i\|=\alpha_i$ for $1\leq i\leq n-1$.
In other words, these are the candidates, up to a scaling, for the first $n-1$ vectors, out of $n$, whose existence is claimed by Proposition \ref{proposition necessity}.
Hence, we call $\X{\setminus}M'$ the {\em space of solution candidates}.

\subsubsection{The test map}

Let $\varphi\colon\X\to\R^d$ be the restriction of the linear map
\[
(\R^d)^{n-1}\to \R^d,
\qquad
(x_1,\ldots,x_{n-1})\mapsto\sum_{1\leq i\leq n-1}x_i,
\]
onto the product of spheres $\X=S_1\times\dots\times S_{n-1}$.
The first property of the map $\varphi$ we show is that its image $\varphi(\X)$ intersects the sphere $S_n\subseteq\R^d$.

\begin{lemma}
\label{lemma : intersection}
$\varphi(\X)\cap S_n\neq \emptyset$.
\end{lemma}
\begin{proof}
Let $v\in S^{d-1} = S_n$ be an arbitrary unit vector in $\R^d$.
Without loss of generality we can assume that the values $\alpha_i$ are ordered in the non-decreasing order $0<\alpha_1\leq \alpha_2\leq \dots\leq \alpha_n$.
Consider the vectors
\[
w_0\ :=\ \varphi\big(\alpha_1v,\dots,\alpha_{n-1}v)=(\alpha_1+\dots+\alpha_{n-1}\big)v,
\]
and
\begin{align*}
w_1\ :=\ & \varphi(-\alpha_1v,\alpha_2v,\dots,(-1)^{n-2}\alpha_{n-2}v,(-1)^{n-1}\alpha_{n-1}v)\\
=\ & \big(-\alpha_1+\alpha_2-\dots+(-1)^{n-2}\alpha_{n-2}+(-1)^{n-1}\alpha_{n-1}\big)v.
\end{align*}
From the assumption \eqref{system of inequalities} we can estimate the norm of $w_0$ by:
\[
\|w_0\|=\alpha_1+\dots+\alpha_{n-1}>\alpha_n.
\]
In the case of the vector $w_1$ we have that
\[
\|w_1\|=|-\alpha_1+\alpha_2-\dots+(-1)^{n-2}\alpha_{n-2}+(-1)^{n-1}\alpha_{n-1}|<\alpha_n.
\]
Indeed, if $n$ is odd then, from the assumption on the ordering of values of $\alpha_i$'s, we see that
\[
\|w_1\|=|-\alpha_1+\alpha_2-\dots-\alpha_{n-2}+\alpha_{n-1}|
= (\alpha_2-\alpha_1)+\cdots+ (\alpha_{n-1}-\alpha_{n-2}).
\]
Consider the differences  $\alpha_2-\alpha_1$,\dots, $\alpha_{n-1}-\alpha_{n-2}$ as lengths of the intervals $[\alpha_1,\alpha_2]$,\dots, $[\alpha_{n-2},\alpha_{n-1}]$, all contained in the interval $[0,\alpha_n]$.
Since these intervals are either disjoint or intersect in a boundary point, and in addition $\alpha_0>0$ we see that $\|w_1\|<\alpha_n$.
On the other hand, if $n$ is even then,
\begin{align*}
\|w_1\|&=|-\alpha_1+\alpha_2-\dots+(-1)^{n-2}\alpha_{n-2}+(-1)^{n-1}\alpha_{n-1}|\\&=|(\alpha_2-\alpha_1)+\cdots+ (\alpha_{n-2}-\alpha_{n-3})-\alpha_{n-1}|.
\end{align*}
Using the same reasoning as in the case when us $n$ odd, now for the sequence of values $\alpha_1,\dots,\alpha_{n-1}$, we have that $0\leq (\alpha_2-\alpha_1)+\cdots+ (\alpha_{n-2}-\alpha_{n-3})<\alpha_{n-1}$.
Consequently, $\|w_1\|<\alpha_n$.

\medskip
The complement $\R^d{\setminus}S_n$ is a disconnected space with two connected components
\[
C_0:=\{ w\in\R^d \ : \ \| w\|>\alpha_n\}
\qquad\text{and}\qquad
C_1:=\{ w\in\R^d \ : \ \| w\|<\alpha_n\},
\]
which are also its path-connected components.
We have seen that $w_0\in C_0$ and $w_1\in C_1$.
Thus, $\varphi(\X)\cap C_0\neq\emptyset$ and $\varphi(\X)\cap C_1\neq\emptyset$.
The map $\varphi$ is continuous, $\X$ is path-connected, and therefore its image $\varphi(\X)$ is also path-connected.
Consequently, $\varphi(\X)$ cannot be contained in $\R^d{\setminus}S_n$ and intersect non-trivially both connected components.
Hence, $\varphi(\X)\cap S_n\neq 0$.
\end{proof}

\medskip
As a direct consequence of the previous lemma we get the following corollary which would already give us Proposition \ref{proposition necessity} if we did not require the vectors spanning the ambient $\R^d$ to be pairwise distinct.
For this reason we call the map $\varphi$, as well as its restrictions, a {\em test map} for our problem.

\begin{corollary}
There exists a collection of  unit vectors $(u_1,\dots, u_n)\in (\R^d)^n$ such that \[
\qquad\quad \sum_{1\leq i\leq n} \alpha_i u_i = 0.
\]
\end{corollary}
\begin{proof}
As we have seen in  Lemma \ref{lemma : intersection} the intersection $\varphi(\X)\cap S_n$ is non-empty.
This implies that there exists $(x_1,\ldots,x_{n-1})\in \X$ with $\|\varphi (x_1,\ldots,x_{n-1})\|=\alpha_n$.
Taking
\[
u_1:=\tfrac{x_1}{\|x_1\|},\ \dots,\ u_{n-1}:=\tfrac{x_{n-1}}{\|x_{n-1} \|},\ u_n:=-\tfrac{\varphi (x_1,\ldots,x_{n-1})}{\|\varphi (x_1,\ldots,x_{n-1})\|}
\]
we get that $\sum_{1\leq i\leq n} \alpha_i u_i = 0$.
\end{proof}
\medskip
In order, find the collection of vectors $(u_1,\dots, u_n)$ which satisfy the desired properties we study further the map $\varphi$ and its restrictions. Recall from (\ref{definition corank 1 matrices}) and (\ref{definition corank 2 matrices}) that we denote by $M'$ the variety of $d\times (n-1)$ matrices of rank at most $d-1$ and by
$M''$ the variety of matrices of rank at most $d-2$.
Let $\XX$ be the open $d(n-1)$-dimensional manifold given by
\[
\XX:=
\begin{cases}
	\X{\setminus}M'', &\quad d\geq 3, \ n=d+1,\\
	\X{\setminus}M',  &\quad n\geq d+2,
\end{cases}
\]
and let us denote by $\psi\colon\XX\to\R^d$ the restriction of $\varphi$ to $\XX$, that is $\psi:=\varphi|_{\XX}$.

\begin{lemma}
\label{lemma : regular value}
Every point contained in the image of $\psi$ is a regular value of $\psi$.
\end{lemma}
\begin{proof}
Consider first the case of a simplex, that is $d\geq 3$ and $n=d+1$.
Pick $w\in\im(\psi)=\varphi(\X{\setminus}M'')$ and $x=(x_1,\dots,x_{n-1})\in\psi^{-1}(\{w\})\subseteq \XX$.
From Lemma \ref{lemma : transversal} we have that $\dim(\X\cap M'')<\dim(\X)$ because $\X\pitchfork M_r$ for every $1\leq r\leq n-1$.
Hence,
\begin{equation}
\label{eq : 101}
\tau_x\XX=\tau_x(\X{\setminus}M'')=\tau_x\X= \big\{y=(y_1,\ldots,y_{n-1})\in (\R^d)^{n-1} \ : \ \langle x_i,y_i\rangle=0,\, 1\leq i\leq n-1\big\}.
\end{equation}
Then, since $\psi$ is a restriction of a linear map we have that the differential $(D\psi)_x\colon \tau_x\XX\to \tau_w\R^d$ is given by $y=(y_1,\ldots,y_{n-1})\mapsto y_1+\dots+y_{n-1}$.

\medskip
Now, we show that  $(D\psi)_x$ is surjective; for relevant definitions see for example \cite{Lee2013}.
For this we denote by $H_i:=(\spann\{x_i\})^{\perp}$, $1\leq i\leq n-1$, the orthogonal complement hyperplane of the vector $x_i$ in~$\R^d$.
Then,
\[
\im (D\psi)_x=H_1+\dots+H_{n-1}\subseteq \tau_w\R^d.
\]
Since $x\in\psi^{-1}(\{w\})\subseteq \XX=\X{\setminus}M''$ and $d\geq 3$ we have that
$
\rank (x)=\rank [x_1 \cdots x_{n-1}]\geq d-1\geq 2
$,
implying that not all of the hyperplanes $H_1,\dots,H_{n-1}$ coincide.
Consequently, $\dim(H_1+\dots+H_{n-1})=d$ and so $\im (D\psi)_x=H_1+\dots+H_{n-1}=\tau_w\R^d$, implying that $(D\psi)_x$ is surjective.
Since $x$ was chosen arbitrary in $\psi^{-1}(\{w\})$ we have proved that $w$ is a regular value of $\psi$.

\medskip
In the case $n\geq d+2$ we proceed in the same way as in the case of a simplex with the only difference that now $x\in\psi^{-1}(\{w\})\subseteq \XX=\X{\setminus}M'$ implies that $\rank (x)=d\geq 2$.
Thus, once more we have that $w$ is a regular value of the map $\psi$.
\end{proof}

\medskip
Furthermore, an analogue of Lemma \ref{lemma : intersection} for the map $\psi$ holds.

\begin{lemma}
\label{lemma : intersection 02}
$\psi(\XX)\cap S_n\neq \emptyset$.
\end{lemma}
\begin{proof}
From Lemma \ref{lemma : transversal} follows that both, $M''$ in the case $d\geq 3$, $n=d+1$, and $M'$ in the case $d\geq 2$, $n\geq d+2$, are of codimension at least $2$ in $\X$.
Consequently, the complement of $\XX$ in $\X$ is path-connected.

\medskip
Let $a:=(\alpha_1v,\dots,\alpha_{n-1}v)$ and $b:=(-\alpha_1v,\alpha_2v,\dots,(-1)^{n-2}\alpha_{n-2}v,(-1)^{n-1}\alpha_{n-1}v)$ be points in $\X$ considered in the proof of Lemma \ref{lemma : intersection}.
Then there exists a path $\gamma\colon [0,1]\to \X$ from $a=\gamma(0)$ to $b=\gamma(1)$ with the property that $\gamma((0,1))\subseteq \XX$. As in the proof of Lemma \ref{lemma : intersection}, we set $C_0:=\{ w\in\R^d \ : \ \| w\|>\alpha_n\}$ and $C_1:=\{ w\in\R^d \ : \ \| w\|<\alpha_n\}$.
Since $w_0=\varphi(a) \in C_0$ and $w_1=\varphi(b) \in C_1$, and $C_0$ and $C_1$ are the path-connected components of the complement $\R^d{\setminus}S_n$, we can find $t\in (0,1)$ with the property that~$\varphi(\gamma(t))\in S_n$.
Hence, $\psi(\XX)\cap S_n\neq \emptyset$.
\end{proof}

\subsubsection{Proof of Proposition \ref{proposition necessity}}
We are going to prove the existence of a collection of pairwise distinct unit vectors $(u_1,\dots, u_n)\in (\R^d)^n$ with the property that \[
\spann\{u_1,\dots, u_n\}=\R^d\qquad\quad\text{and}\qquad\quad \sum_{1\leq i\leq n} \alpha_i u_i = 0.
\]
For that we show the existence of another  collection of vectors $x=(x_1,\dots,x_{n-1})\in\X$ with the property that
 \begin{equation}
 \label{proof of prop: 01}
 x\notin M'
 \qquad\text{and}\qquad
 \varphi(x)\in S_n
 \qquad\text{and}\qquad
 \big(\tfrac{x_1}{\|x_1\|},\dots,\tfrac{x_{n-1}}{\|x_{n-1}\|},-\tfrac{\varphi(x)}{\| \varphi(x)\|}\big)\in\conf (S^{d-1},n),
 \end{equation}
 where $\conf (S^{d-1},n)\subseteq (S^{d-1})^n$ denotes the ordered configuration space of $n$ pairwise distinct points on the sphere $S^{d-1} = \{x\in\mathbb R^d : \langle x,x\rangle = 1\}$.
 If $x\in\X$ does not satisfy conditions \eqref{proof of prop: 01} then, either
 \begin{compactenum}[\rm (i)]
 \item  $x\in  M'$, or
 \item  $\varphi(x)\notin S_n$, or
 \item 	$\tfrac{x_i}{\|x_i\|}=\tfrac{x_j}{\|x_j\|} \Leftrightarrow \alpha_jx_i=\alpha_ix_j$ for some $1\leq i<j\leq n-1$, or
 \item  $\tfrac{x_i}{\|x_i\|}=-\tfrac{\varphi(x)}{\| \varphi(x)\|}$ for some $1\leq i\leq n-1$.
 \end{compactenum}
We will find a point $x\in\X$ satisfying none of the properties (i)--(iv).
In the following, as before, for the collection of vectors $x$ we also use the matrix notation $X$ when convenient.

\medskip
According to the Lemma \ref{lemma : regular value} and Lemma \ref{lemma : intersection 02} there exists a regular value $w$ of the map $\psi$ which belongs to the sphere $S_n$, that is $w\in S_n$.
We proceed by applying different strategies in the case of simplices, and then in the case of arbitrary polytope. Recall from (\ref{definition rank r matrices}) that $M_r$ denotes the manifold of $d\times (n-1)$ matrices of rank equal to $r$.

\bigskip
{\sc (1)}
We start with the case of simplices.
For this, let $d\geq 3$ and $n=d+1$ be integers.
Take an arbitrary $x \in\psi^{-1}(\{w\})$, then condition (ii) does not hold.
If $\rank(x)=d=n-1$, or equivalently $x\in \XX{\setminus}M'$, condition (i) is not satisfied, and conditions (iii) and (iv) do not hold because they contradict the assumption that vectors $x_1,\ldots,x_{n-1}$  are linearly independent.
Thus, if $\rank(x)=d$ the proof is complete, otherwise we assume that $x\in \XX{\setminus}M_d=M_{d-1}$ and show that there exists an another point in~$\psi^{-1}(\{w\})$ whose rank is maximal.

\medskip
We start with $x\in \XX{\setminus}M_d=\X\cap M_{d-1}$, and assume the contrary, that $\psi^{-1}(\{w\})\subseteq M_{d-1}$.
Since $w$ is a regular value of $\psi$ the preimage  $\psi^{-1}(\{w\})$ is a smooth submanifold of $\XX$ of codimension $d$.
Then, the tangent space of $\psi^{-1}(\{w\})$ at the point $x$ is $\tau_x\psi^{-1}(\{w\})=\ker (D\psi)_x$, see for example \cite[Theorem\,A.9]{Buergisser2013}.
Note that the assumption $\psi^{-1}(\{w\})\subseteq M_{d-1}$ implies that $\tau_x\psi^{-1}(\{w\})\subseteq \tau_xM_{d-1}$.
Hence, we consider next the tangent space $\tau_xM_{d-1}$ of $M_{d-1}$ at $x$.
This time we describe it using its normal (bundle) subspace with respect to the ambient $M=(\R^{d})^{n-1}$ and the scalar product we already used, $\langle\langle A,B\rangle\rangle:=\trace (A^t\cdot B)$ for $A,B\in M$.
That is,
\[
\nu_xM_{d-1}=\{Y\in M : (\forall Z\in\tau_xM_{d-1} ) \ \langle\langle Y,Z \rangle \rangle=0\}.
\]
Alternatively, the normal space can be described as follows.
Let us consider the point $x$ as a matrix $X\in M_{d-1}$ and present it as a product $X=U\cdot V^{t}$ where $U\in\Mat_{d\times (d-1)}(\R)$, $V\in\Mat_{d\times (d-1)}(\R)$ and $\rank(U)=\rank(V)=d-1$.
Furthermore, denote by $U^{\perp}\subseteq\R^d$ and $V^{\perp}\subseteq\R^{d}$ the subspaces orthogonal to the column spans of the matrices $U$ and $V$, respectively.
Note than $\dim (U^{\perp})=\dim (V^{\perp})=1$.
We claim that
\begin{equation}\label{Normal space rank r matrices}
\nu_xM_{d-1}=\spann \{u\cdot v^{t} : u\in U^{\perp},  v\in V^{\perp}\}.
\end{equation}
We use the description of $\tau_xM_{d-1}$ in (\ref{Tangent space rank r matrices}) to see that the right-hand side in (\ref{Normal space rank r matrices}) is indeed perpendicular to $\tau_xM_{d-1}$. This shows that the right-hand side of (\ref{Normal space rank r matrices}) is included in  $\nu_xM_{d-1}$.
Equality follows, because $M_{d-1}$ is of codimension $1$ in $\Mat_{d\times d}(\R)$ and so both sides are of dimension $1$.

\medskip
To complete the proof we will find a vector $Y$ which belongs to $\tau_x\psi^{-1}(\{w\})$ but is not in $\tau_xM_{d-1}$, contradicting the assumption that $\tau_x\psi^{-1}(\{w\})\subseteq \tau_xM_{d-1}$.
For that we fix two unit vectors $u_0\in U^{\perp}$ and $v_0\in V^{\perp}$, and set $\varepsilon:=(1,1,\dots,1)\in\R^{n-1}$.
First, we observe that
\[
X^t\cdot u_0=(U\cdot V^t)^t\cdot u_0= V\cdot (U^t\cdot u_0)=0
\]
which in particular implies that $x_i^t\cdot u_0=0$ for all $1\leq i\leq d$.
Next, we consider the vector $Y_{\lambda}:=(\lambda_1u_0,\dots,\lambda_{d}u_0)=u_0\cdot\lambda^t$ for an arbitrary choice of the vector $\lambda:=(\lambda_1,\dots,\lambda_{d}) \in \R^d$.
According to~\eqref{eq : 101}, we have that $Y_{\lambda}\in \tau_x\XX=\tau_x(\X{\setminus}M'')$ because $\langle x_i,\lambda_iu_0\rangle= x_i^t\cdot ( \lambda_i u_0 )= \lambda_i ( x_i^t\cdot u_0)=0$.
On the other hand, if the coordinates of the vector $\lambda$ sum to zero, then $Y_{\lambda}\in \ker (D\psi)_x=\tau_x\psi^{-1}(\{w\})$ because,
\[
(D\psi)_x(Y_{\lambda})=(D\psi)_x(\lambda_1u_0,\dots,\lambda_{d}u_0)=(\lambda_1+\dots+\lambda_{d})u_0=0.
\]
Next, $\psi(x)=X\cdot\varepsilon \neq 0$ because  $\psi(x)\in S_n$, and $X\cdot v_0=(U\cdot V^{t})\cdot v_0=0$, implying the linear independence of $\varepsilon^t$ and $v_0$.
Therefore, we can find a vector $\lambda_0\in \R^{d}$ with the property that $\langle \lambda_0,\varepsilon \rangle= \lambda_0^t \cdot \varepsilon =0$ and  $\langle \lambda_0, v_0 \rangle =\lambda_0^t \cdot v_0\neq 0$ since the orthogonal complements of $\varepsilon$ and $v_0$ do not coincide.
Then we have that
\begin{align*}
\langle \langle Y_{\lambda_0},u_0\cdot v_0^t\rangle\rangle=
\trace\big( Y_{\lambda_0}^t\cdot (u_0\cdot v_0^t)\big)=
\trace\big( Y_{\lambda_0}\cdot (u_0\cdot v_0^t)^t\big)
&= \trace\big( u_0\cdot \lambda_0^t \cdot  v_0 \cdot u_0^t\big)\\
&= (\lambda_0^t \cdot  v_0) \trace (u_0\cdot u_0^t)\\
&= \lambda_0^t \cdot  v_0\neq 0,
\end{align*}
because $u_0$ is chosen to be a unit vector.
Here we use the identity $\trace(A^t\cdot B)=\trace (A\cdot B^t)$.
Thus, the fact $\langle \langle Y_{\lambda_0},u_0\cdot v_0^t\rangle\rangle\neq 0$ implies that $Y_{\lambda_0}\notin \tau_x M_{d-1}$ but $Y_{\lambda_0}\in\tau_x\psi^{-1}(\{w\})$, and the contradiction we announced is reached.
This completes the proof of Proposition \ref{proposition necessity} in the case $d\geq 3$ and $n=d+1$.

\bigskip
{\sc (2)}
Now we consider the case of polytopes.
Let $d\geq 2$ and $n\geq d+2$ be integers.
For every point $x \in\psi^{-1}(\{w\})\subset \XX$ we have that $\rank(x)=d$ so that conditions (i) and (ii) do not hold.
Hence, to prove the proposition, we need to find a point in $\psi^{-1}(\{w\})$ which does not satisfy either condition (iii) or condition (iv).

\medskip
Let $x=(x_1,\dots,x_{n-1})\equiv X\in \psi^{-1}(\{w\})$ be fixed and let $x$ satisfy one of the conditions (iii) or~(iv).
Consequently, there is a linear dependence between vectors $x_1,\dots,x_{n-1}$ which can be encoded by $X\cdot \ell=0$ for some concrete non-zero vector $\ell\in\R^{n-1}$.
Consider the linear subspace of $M$:
\[
\mathcal{L}:=\{Y\in M : Y\cdot\ell=0\}=\{Y\in M : \langle\langle Y, e_i\cdot \ell^t\rangle\rangle=0, \ 1\leq i\leq d\}.
\]
Here $e_1,\dots, e_d$ denotes the standard basis of $\R^d$, and $f_1,\dots, f_{n-1}$ is the standard basis of $\R^{n-1}$.
Then, $X\in \mathcal{L}$ and $\tau_X\mathcal{L}=\mathcal{L}$.
According to Lemma \ref{lemma : transversal} we have that $\dim(\X\cap M')<\dim(\X)$, and therefore $\tau_X\XX=\tau_X(\X{\setminus}M')=\tau_X\X$.
From the equality $\ker (D\psi)_x=\tau_x\psi^{-1}(\{w\})$ we get that
\begin{align*}
\tau_X\psi^{-1}(\{w\})
&=
\{y=(y_1,\dots,y_{n-1})\equiv Y\in M : \langle x_j,y_j\rangle = 0, \ 1\leq j\leq n-1,  \
y_1+\dots+y_{n-1}=0\} \\
&=\{ Y\in M :  \langle\langle Y, x_j\cdot f_j^t\rangle\rangle=0, \ 1\leq j\leq n-1, \  \langle\langle Y, e_i\cdot f^t\rangle\rangle=0, \ 1\leq i\leq d\},
\end{align*}
where $f:=\sum_{1\leq j\leq n-1}f_j$.
Let $a:=(a_1,\dots,a_d)\neq 0$ be a non-zero vector in $\R^d$ which is not a multiple of any of the vectors $x_1,\dots,x_{n-1},\psi(x)$.
The vector $\sum_{1\leq i\leq d} a_i(e_i\cdot \ell^t)\in \mathcal{L}^{\perp}$ belongs to the orthogonal complement of $\mathcal{L}$.
If $\tau_X\psi^{-1}(\{w\})\subseteq \mathcal{L}$, then on the level of orthogonal complements the inclusion changes direction, that is $\mathcal{L}^{\perp}\subseteq (\tau_X\psi^{-1}(\{w\}))^{\perp}$.
Consequently,
\begin{equation}
\label{eq 102}
\sum_{1\leq i\leq d} a_i(e_i\cdot \ell^t)=
\sum_{1\leq i\leq d} b_i(e_i\cdot f^t)+
\sum_{1\leq j\leq n-1}c_j (x_j\cdot f_j^t)
\end{equation}
for some vectors $b:=(b_1,\dots,b_d)$ and $c:=(c_1,\dots,c_{n-1})$.
The relation \eqref{eq 102} simplifies into:
\begin{equation}
\label{eq 103}
a\cdot \ell^t=b\cdot f^t+X\cdot\diag(c).
 \end{equation}
We prove that the relation \eqref{eq 103} cannot hold.

\medskip
Assume the opposite, that equality \eqref{eq 103} holds.
In the first step, we show that the vectors $a$ and $b$ are linearly independent.
If $b=\lambda a$ for some $\lambda \in \R$, it implies that $a \cdot (\ell^t-\lambda\cdot f^t)=X\cdot\diag(c)$.
Since
\begin{align*}
\ell&=\alpha_pf_q-\alpha_qf_p,\quad\quad\text{for some }1\leq p<q\leq n-1, or\\
\ell&=\alpha_pf+\alpha_nf_p, \ \quad\quad\text{for some }1\leq p\leq n-1,
\end{align*}
we have that $\ell\notin\spann\{f\}$.
Then there exists a coordinate of $\ell$ different from $\lambda$, that is $\ell_r\neq \lambda$ for some $1\leq r\leq n-1$.
Now the $r$th coordinate of the equality $a \cdot (\ell^t-\lambda\cdot f^t)=X\cdot\diag(c)$ implies that $a=\frac{c_r}{\ell_r-\lambda}x_j$, a contradiction with the choice of the vector $a$.

\medskip
We conclude that $a$ and $b$ have to be linearly independent.
Moreover, since $\rank(X)=d$ then, $a\cdot \ell^t-b\cdot f^t=X\cdot\diag(c)$ has rank $2$ implying that only two entries of $\diag(c)$ are non-zero.
Without loss of generality we can assume that $c=c_1f_1+c_2f_2$, where $c_1\neq 0$ and $c_2\neq 0$.
Hence,
\[
a\cdot \ell^t-b\cdot f^t=X\cdot\diag(c)=c_1(x_1\cdot f_1^t)+c_2(x_2\cdot f_2^t).
\]
Recall that $n\geq d+2\geq 4$ which implies $n-1\geq 3$.
Therefore, the third coordinate of equation \eqref{eq 103} implies that $\ell_3 a = b + 0$, contradicting the linear independence of $a$ and $b$.

\medskip
Consequently, the equality \eqref{eq 103} cannot hold, which implies that $\mathcal{L}^{\perp}\not\subseteq (\tau_X\psi^{-1}(\{w\}))^{\perp}$, or dually $\tau_X\psi^{-1}(\{w\})\not\subseteq \mathcal{L}$.
Hence, there exists a non-zero tangent vector in $\tau_X\psi^{-1}(\{w\})$ which does not belong to any of the linear spaces defined by conditions (iii) and (iv) for a fixed matrix $X$. Therefore, there is a curve in $\psi^{-1}(\{w\})$ passing through $X$ that contains points which do not satisfy (iii) or (iv) as well.
Any such point satisfies conditions \eqref{proof of prop: 01}, and the proof of the proposition is complete.\qed

\bibliography{references_facet_volumes_of_simplices}{}
\bibliographystyle{amsplain}

\end{document}